\documentclass{conm-p-l}

\usepackage{amssymb}

\usepackage{graphicx}
\usepackage{subfigure}
\usepackage{float}
\usepackage{scalefnt}
\usepackage{url}
\usepackage{color}

\usepackage{}

\newtheorem{theorem}{Theorem}[section]
\newtheorem{lemma}[theorem]{Lemma}

\newenvironment{gh}{\color{black}}

\theoremstyle{definition}
\newtheorem{definition}[theorem]{Definition}

\theoremstyle{remark}
\newtheorem{remark}[theorem]{Remark}
\newtheorem{proposition}[theorem]{Proposition}

\numberwithin{equation}{section}

\begin{document}

 \title[SOS Basis Pursuit with Linear and Second Order Cone Programming]{Sum of Squares Basis Pursuit \\with Linear and Second Order Cone Programming}


\author{Amir Ali Ahmadi}
\address{ORFE, Princeton University, Sherrerd Hall, Charlton Street, Princeton 08540, USA}
\curraddr{}
\email{a\_a\_a@princeton.edu}
\thanks{}

\author{Georgina Hall}
\address{ORFE, Princeton University, Sherrerd Hall, Charlton Street, Princeton 08540, USA}
\curraddr{}
\email{gh4@princeton.edu}
\thanks{The authors are partially supported by the Young Investigator Program award of the AFSOR}


\date{}

\begin{abstract}
We devise a scheme for solving an iterative sequence of linear programs (LPs) or second order cone programs (SOCPs) to approximate the optimal value of {\gh any semidefinite program (SDP) or sum of squares (SOS) program}. The first LP and SOCP-based bounds in the sequence come from the recent work of Ahmadi and Majumdar on diagonally dominant sum of squares (DSOS) and scaled diagonally dominant sum of squares (SDSOS) polynomials. We then iteratively improve on these bounds by pursuing better bases in which more relevant SOS polynomials admit a DSOS or SDSOS representation. Different interpretations of the procedure from primal and dual perspectives are given. While the approach is applicable to {\gh SDP} relaxations of general polynomial programs, we apply it to two problems of discrete optimization: the maximum independent set problem and the partition problem. We further show that some completely trivial instances of the partition problem lead to strictly positive polynomials on the boundary of the sum of squares cone and hence make the SOS relaxation fail.
\end{abstract}

\maketitle

\section{Introduction}
In recent years, semidefinite programming~\cite{VaB:96} and sum of squares optimization~\cite{sdprelax, lasserre_moment, NesterovSquared} have proven to be powerful techniques for tackling a diverse set of problems in applied and computational mathematics. The reason for this, at a high level, is that several fundamental problems arising in discrete and polynomial optimization~\cite{Monique_sos_moments_survey, Stability_number_SOS, OR_isos_OptLetters} or the theory of dynamical systems~\cite{PhD:Parrilo, PositivePolyInControlBook, AAA_PhD} can be cast as linear optimization problems over the cone of nonnegative polynomials. This observation puts forward the need for efficient conditions on the coefficients $c_\alpha\mathrel{\mathop:}=c_{\alpha_1,\ldots,\alpha_n}$ of a multivariate polynomial $$p(x)=\sum_{\alpha } c_{\alpha_1,\ldots,\alpha_n} x_1^{\alpha_1}\ldots x_n^{\alpha_n}$$
that ensure the inequality $p(x)\geq 0,$ for all $x\in\mathbb{R}^n$. If $p$ is a quadratic function, $p(x)=x^TQx+2c^Tx+b,$ then nonnegativity of $p$ is equivalent to the $(n+1)\times (n+1)$ symmetric matrix $$\begin{pmatrix}
Q & c\\ c^T & b
\end{pmatrix}$$
being positive semidefinite and this constraint can be imposed by semidefinite programming. For higher degrees, however, imposing nonnegativity of polynomials is in general an intractable computational task. In fact, even checking if a given quartic polynomial is nonnegative is NP-hard~\cite{copositivity_NPhard}. A particularly popular and seemingly powerful sufficient condition for a polynomial $p$ to be nonnegative is for it to decompose as a sum of squares of other polynomials:
$$p(x)=\sum_i q_i^2(x).$$

This condition is attractive for several reasons. From a computational perspective, for fixed-degree polynomials, a sum of squares decomposition can be checked (or imposed as a constraint) by solving a semidefinite program of size polynomial in the number of variables. From a representational perspective, such a decomposition \emph{certifies} nonnegativity of $p$ in terms of an easily verifiable algebraic identity. From a practical perspective, the so-called ``sum of squares relaxation'' is well-known to produce powerful (often exact) bounds on optimization problems that involve nonnegative polynomials; see, e.g.,~\cite{Minimize_poly_Pablo}. The reason for this is that constructing examples of nonnegative polynomials that are not sums of squares in relatively low dimensions and degrees seems to be a difficult task\footnote{See~\cite{Reznick} for explicit examples of nonnegative polynomials that are not sums of squares.}, especially when additional structure arising from applications is required.

We have recently been interested in leveraging the attractive features of semidefinite programs (SDPs) and sum of squares (SOS) programs, while solving much simpler classes of convex optimization problems, namely \emph{linear programs} (LPs) and \emph{second order cone programs} (SOCPs). {\gh Such a research direction} can potentially lead to a better understanding of the relative power of different classes of convex relaxations. It also has obvious practical motivations as simpler convex programs come with algorithms that have better scalability and improved numerical conditioning properties. This paper is a step in this research direction. We present a scheme for solving a sequence of LPs or SOCPs that provide increasingly accurate approximations to the optimal value and the optimal solution of a semidefinite (or a sum of squares) program. With the algorithms that we propose, one can use one of many mature LP/SOCP solvers such as~\cite{cplex, gurobi, mosek}, including simplex-based LP solvers, to obtain reasonable approximations to the optimal values of these more difficult convex optimization problems.


The intuition behind our approach is easy to describe with a contrived example.
%
%
%
Suppose we would like to show that the degree-4 polynomial 
\begin{equation}\nonumber
\begin{array}{lll}
p(x)&=&x_1^4-6x_1^3x_2+2x_1^3x_3+6x_1^2x_3^2+9x_1^2x_2^2-6x_1^2x_2x_3-14x_1x_2x_3^2+4x_1x_3^3
\\ \ &\ &+5x_3^4-7x_2^2x_3^2+16x_2^4
\end{array}
\end{equation}
has a sum of squares decomposition. One way to do this is to attempt to write $p$ as $${\gh p(x)=z^T(x)Qz(x),}$$
where
\begin{equation}\label{eq:z(x).quadatic}
z(x)=(x_1^2,x_1x_2,x_2^2,x_1x_3,x_2x_3,x_3^2)^T
\end{equation}  
is the standard (homogeneous) monomial basis of degree 2 and the matrix $Q$, often called the \emph{Gram matrix}, is symmetric and positive semidefinite. The search for such a $Q$ can be done with semidefinite programming; one feasible solution e.g. is as follows:

$$Q=\begin{pmatrix}
1 & -3 & 0 & 1 & 0 & 2 \\
-3 & 9 & 0 & -3 & 0 & -6 \\
0 & 0 & 16 & 0 & 0 & -4 \\
1 & -3 & 0 & 2 & -1 & 2 \\
0 & 0 & 0 & -1 & 1 & 0 \\
2 & -6 & 4 & 2 & 0 & 5 
\end{pmatrix}.$$
Suppose now that instead of the basis $z$ in (\ref{eq:z(x).quadatic}), we pick a different basis 
\begin{equation}\label{eq:ztilde}
{\gh \tilde{z}(x)=(2x_1^2-6x_1x_2+2x_1x_3+2x_3^2, \quad x_1x_3-x_2x_3, \quad x_2^2-\frac{1}{4}x_3^2)^T.}
\end{equation}
With this new basis, we can get a sum of squares decomposition of $p$ by writing it as 
$$p(x)=\tilde{z}^T(x) \begin{pmatrix}
\frac{1}{2} & 0 & 0 \\ 0& 1 & 0 \\ 0 & 0 & 4.
\end{pmatrix}  \tilde{z}(x).$$
In effect, by using a better basis, we have simplified the Gram matrix and made it diagonal. When the Gram matrix is diagonal, its positive semidefiniteness can be imposed as a \emph{linear} constraint (diagonals should be nonnegative).

Of course, the catch here is that we do not have access to the magic basis $\tilde{z}(x)$ in (\ref{eq:ztilde}) a priori. Our goal will hence be to ``pursue'' this basis (or other good bases) by starting with an arbitrary basis (typically the standard monomial basis), and then iteratively improving it by solving a sequence of LPs or SOCPs and performing some efficient matrix decomposition tasks in the process. Unlike the intentionally simplified example we gave above, we will not ever require our Gram matrices to be diagonal. This requirement is too strong and would frequently lead to our LPs and SOCPs being infeasible. The underlying reason for this is that the cone of diagonal matrices is not full dimensional in the cone of positive semidefinite matrices. Instead, we will be after bases that allow the Gram matrix to be \emph{diagonally dominant} or \emph{scaled diagonally dominant} (see Definition~\ref{def:dd.sdd}). The use of these matrices in polynomial optimization has recently been proposed by Ahmadi and Majumdar~\cite{iSOS_journal, dsos_ciss14}. We will be building on and improving upon their results in this paper.

\subsection{Organization of this paper}
The organization of the rest of the paper is as follows. In Section~\ref{sec:prelims}, we introduce some notation and briefly review the concepts of ``dsos and sdsos polynomials'' which are used later as the first step of an iterative algorithm that we propose in Section~\ref{sec:basis.pursuit}. In this section, we explain how we inner approximate semidefinite (Subsection~\ref{subsec:InnerApprox}) and sum of squares (Subsection~\ref{subsec:polyopt}) cones with LP and SOCP-based cones by iteratively changing bases. In Subsection~\ref{subsec:Corners}, we give a different interpretation of our LPs in terms of their corner description as opposed to their facet description. Subsection~\ref{subsec:OuterApprox} is about duality, which is useful for iteratively outer approximating semidefinite or sum of squares cones.

In Section~\ref{sec:StableSet}, we apply our algorithms to the Lov\'{a}sz semidefinite relaxation of the maximum stable set problem. It is shown numerically that our LPs and SOCPs converge to the SDP optimal value in very few iterations and outperform some other well-known LP relaxations on a family of randomly generated examples. In Section~\ref{sec:Partition}, we consider the partition problem from discrete optimization. As opposed to the stable set problem, the quality of our relaxations here is rather poor. In fact, even the sum of squares relaxation fails on some completely trivial instances. We show this empirically on random instances, and formally prove it on one representative example (Subsection~\ref{subsec:sos.refutability}). The reason for this failure is existence of a certain family of quartic polynomials that are nonnegative but not sums of squares.


\section{Preliminaries}\label{sec:prelims}
We denote the set of real symmetric $n\times n$ matrices by $S_n$. Given two matrices $A$ and $B$ in $S_n$, their standard matrix inner product is denoted by $A\cdot B := \sum_{i,j}A_{ij}B_{ij} = \mbox{ Trace}(AB)$.
A symmetric matrix $A$ is \emph{positive semidefinite} (psd) if $x^TAx\geq 0$ for all $x\in\mathbb{R}^n$; this will be denoted by the standard notation $A\succeq 0$, and our notation for the set of $n\times n$ psd matrices is $P_n$. We say that $A$ is \emph{positive definite} (pd) if $x^TAx> 0$ for all $x\neq 0$.
Any psd matrix $A$ has an upper triangular Cholesky factor $U=\text{chol}(A)$ satisfying $A=U^TU$. When $A$ is pd, the Cholesky factor is unique and has positive diagonal entries.
For a cone of matrices in $S_n$, we define its dual cone $\mathcal{K}^*$ as $\{Y \in S_n: Y\cdot X \geq 0, \ \forall X \in \mathcal{K}\}$.

For a vector variable $x \in \mathbb{R}^n$ and a vector {\gh $s \in \mathbb{Z}^n_+$, let a monomial in $x$ be denoted as $x^s= \Pi_{i=1}^n x_i^{s_i}$ which by definition has degree $\sum_{i=1}^n s_i$.}
A polynomial is said to be \emph{homogeneous} or a \emph{form} if all of its monomials have the same degree. A form $p(x)$ in $n$ variables is nonnegative if $p(x)\geq 0$ for all $x\in\mathbb{R}^n$, or equivalently for all $x$ on the unit sphere in $\mathbb{R}^n$. The set of nonnegative (or positive semidefinite) forms in $n$ variables and degree $d$ is denoted by $PSD_{n,d}$. A form $p(x)$ is a \emph{sum of squares} (sos) if it can be written as $p(x)=\sum_{i=1}^r q_i^2(x)$ for some forms $q_1,\ldots,q_r$. The set of sos forms in $n$ variables and degree $d$ is denoted by $SOS_{n,d}$. We have the obvious inclusion $SOS_{n,d}\subseteq PSD_{n,d}$, which is strict unless $d=2$, or $n=2$, or $(n,d)=(3,4)$~\cite{Hilbert_1888}. Let $z(x,d)$ be the vector of all monomials of degree exactly $d$; it is well known that a form $p$ of degree $2d$ is sos if and only if it can be written as $p(x)=z^T(x,d)Qz(x,d)$, for some psd matrix $Q$~\cite{sdprelax, PhD:Parrilo}. An SOS optimization problem is the problem of minimizing a linear function over the intersection of the convex cone $SOS_{n,d}$ with an affine subspace. The previous statement implies that SOS optimization problems can be cast as semidefinite programs.


%

\subsection{DSOS and SDSOS optimization}
\label{subsec:dsos.sdsos}

In recent work, Ahmadi and Majumdar introduce more scalable alternatives to SOS optimization that they refer to as \emph{DSOS and SDSOS programs}~\cite{iSOS_journal,dsos_ciss14}\footnote{The work in~\cite{iSOS_journal} is currently in preparation for submission; the one in~\cite{dsos_ciss14} is a shorter conference version of~\cite{iSOS_journal} which has already appeared. The presentation of the current paper is meant to be self-contained.}. Instead of semidefinite programming, these optimization problems can be cast as linear and second order cone programs respectively. Since we will be building on these concepts, we briefly review their relevant aspects to make our paper self-contained.



The idea in~\cite{iSOS_journal, dsos_ciss14}  is to replace the condition that the Gram matrix $Q$ be positive semidefinite with stronger but cheaper conditions in the hope of obtaining more efficient inner approximations to the cone $SOS_{n,d}$. Two such conditions come from the concepts of \emph{diagonally dominant} and \emph{scaled diagonally dominant} matrices in linear algebra. We recall these definitions below.

\begin{definition}\label{def:dd.sdd}
A symmetric matrix $A$ is \emph{diagonally dominant} (dd) if $a_{ii} \geq \sum_{j \neq i} |a_{ij}|$ for all $i$. We say that $A$ is \emph{scaled diagonally dominant} (sdd) if there exists a diagonal matrix $D$, with positive diagonal entries, which makes $DAD$ diagonally dominant.
\end{definition}

We refer to the set of $n \times n$ dd (resp. sdd) matrices as $DD_n$ (resp. $SDD_n$). The following inclusions are a consequence of Gershgorin's circle theorem {\gh\cite{gersh}}:
$$DD_n\subseteq SDD_n\subseteq P_n.$$

Whenever it is clear from the context, we may drop the subscript $n$ from our notation. We now use these matrices to introduce the cones of ``dsos'' and ``sdsos'' forms which constitute special subsets of the cone of sos forms. We remark that in the interest of brevity, we do not give the original definition of dsos and sdsos polynomials as it appears in~\cite{iSOS_journal} (as sos polynomials of a particular structure), but rather an equivalent characterization of them that is more useful for our purposes. The equivalence is proven in~\cite{iSOS_journal}. 


\begin{definition}[\cite{iSOS_journal,dsos_ciss14}] \label{def:dsos.sdsos.rdsos.rsdsos}
	Recall that $z(x,d)$ denotes the vector of all monomials of degree exactly $d$. A form $p(x)$ of degree $2d$ is said to be
	
	\begin{itemize}
		\item  \emph{diagonally-dominant-sum-of-squares} (dsos) if it admits a representation as $p(x)=z^T(x,d)Qz(x,d)$, where $Q$ is a dd matrix.
		\item  \emph{scaled-diagonally-dominant-sum-of-squares} (sdsos) if it admits a representation as $p(x)=z^T(x,d)Qz(x,d)$, where $Q$ is an sdd matrix.
	\end{itemize}
\end{definition}

The definitions for non-homogeneous polynomials are exactly the same, except that we replace the vector of monomials of degree exactly $d$ with the vector of monomials of degree $\leq d$. We observe that a quadratic form $x^TQx$ is dsos/sdsos/sos if and only if the matrix $Q$ is dd/sdd/psd. Let us denote the cone of forms in $n$ variables and degree $d$ that are dsos and sdsos by $DSOS_{n,d}$, $SDSOS_{n,d}$. The following inclusion relations are straightforward: $$DSOS_{n,d}\subseteq SDSOS_{n,d}\subseteq SOS_{n,d}\subseteq PSD_{n,d}.$$




From the point of view of optimization, our interest in all of these algebraic notions stems from the following theorem.

\begin{theorem}[\cite{iSOS_journal,dsos_ciss14}] For any fixed $d$, optimization over the cones $DSOS_{n,d}$ (resp. $SDSOS_{n,d}$) can be done with linear programming (resp. second order cone programming) of size polynomial in $n$.
\end{theorem}

The ``LP part'' of this theorem is not hard to see. The equality $p(x)=z^T(x,d)Qz(x,d)$ gives rise to linear equality constraints between the coefficients of $p$ and the entries of the matrix $Q$ (whose size is $\sim n^{\frac{d}{2}}\times n^{\frac{d}{2}}$ and hence polynomial in $n$ for fixed $d$). The requirement of diagonal dominance on the matrix $Q$ can also be described by linear inequality constraints on $Q$. The ``SOCP part'' of the statement comes from the fact, shown in~\cite{iSOS_journal}, that a matrix $A$ is sdd if and only if it can be expressed as 	\begin{align}\label{eq:SDSOS.def}A = \sum_{i< j} M_{2 \times 2}^{ij},\end{align}
where each $ M_{2 \times 2}^{ij}$ is an $n\times n$ symmetric matrix with zeros everywhere except for four entries $M_{ii}, M_{ij}, M_{ji}, M_{jj}$, which must make the $2\times 2$ matrix $\begin{bmatrix} M_{ii} & M_{ij}  \\ M_{ji} & M_{jj} \end{bmatrix}$ symmetric and positive semidefinite. These constraints are \emph{rotated quadratic cone} constraints and can be imposed using SOCP~\cite{socp_alizadeh_goldfarb, socp_boyd}:
$$M_{ii}\geq 0, ~\Bigl\lvert\Bigl\lvert\begin{pmatrix}
2M_{ij}\\M_{ii}-M_{jj}
\end{pmatrix}\Bigl\lvert\Bigl\lvert \leq M_{ii}+M_{jj}.$$

We refer to linear optimization problems over the convex cones $DSOS_{n,d}$, $SDSOS_{n,d}$, and $SOS_{n,d}$ as DSOS programs, SDSOS programs, and SOS programs respectively. In general, quality of approximation decreases, while scalability increases, as we go from SOS to SDSOS to DSOS programs. What we present next can be thought of as an iterative procedure for moving from DSOS/SDSOS relaxations towards SOS relaxations without increasing the problem size in each step.


\section{Pursuing improved bases}\label{sec:basis.pursuit}

Throughout this section, we consider the standard SDP
\begin{equation} \label{eq:genericSDP}
\begin{aligned}
SOS^* &\mathrel{\mathop{:}}=\min_{X\in S_n} C \cdot X\\
&\text{s.t. } A_i \cdot X=b_i, i=1,\ldots,m,\\
&\quad \quad \quad \ \ X \succeq 0,
\end{aligned}
\end{equation}
which we assume to have an optimal solution. We denote the optimal value by $SOS^*$ since we think of a semidefinite program as a sum of squares program over quadratic forms (recall that $PSD_{n,2}=SOS_{n,2}$). This is so we do not have to introduce additional notation to distinguish between degree-2 and higher degree SOS programs. The main goal of this section is to construct sequences of LPs and SOCPs that generate bounds on the optimal value of (\ref{eq:genericSDP}). Section  \ref{subsec:InnerApprox} focuses on providing upper bounds on (\ref{eq:genericSDP}) while Section \ref{subsec:OuterApprox} focuses on lower bounds.


\subsection{Inner approximations of the psd cone} \label{subsec:InnerApprox}

To obtain upper bounds on (\ref{eq:genericSDP}), we need to replace the constraint $X \succeq 0$ by a stronger condition. In other words, we need to provide \emph{inner approximations} to the set of psd matrices.

First, let us define a family of cones $$DD(U)\mathrel{\mathop{:}}=\{M \in S_n~|~ M=U^TQU \text{ for some dd matrix } Q \},$$parametrized by an $n \times n$ matrix $U$. Optimizing over the set $DD(U)$ is an LP since $U$ is fixed, and the defining constraints are linear in the coefficients of the two unknowns $M$ and $Q$. Furthermore, the matrices in $DD(U)$ are all psd; i.e., $\forall U,$ $DD(U) \subseteq P_n$.

The iteration number $k$ in the sequence of our LPs consists of replacing the condition $X \succeq 0$ by $X \in DD(U_k)$:
\begin{equation}\label{eq:LPChol}
\begin{aligned}
DSOS_k &\mathrel{\mathop{:}}=\min C \cdot X\\
&\text{s.t. } A_i \cdot X=b_i, ~i=1,\ldots,m,\\
&X \in DD(U_k).
\end{aligned}
\end{equation}
To define the sequence $\{U_k\}$, we assume that an optimal solution $X_k$ to (\ref{eq:LPChol}) exists for every iteration. As it will become clear shortly, this assumption will be implied simply by assuming that only the first LP in the sequence is feasible. The sequence $\{U_k\}$ is then given recursively by
\begin{equation}\label{eq:defUk}
\begin{aligned}
U_0&=I\\
U_{k+1}&=\text{chol}(X_k).
\end{aligned}
\end{equation}

Note that the first LP in the sequence optimizes over the set of diagonally dominant matrices as in the work of Ahmadi and Majumdar~\cite{iSOS_journal,dsos_ciss14}. By defining $U_{k+1}$ as a Cholesky factor of $X_k$, improvement of the optimal value is guaranteed in each iteration. Indeed, as $X_k=U_{k+1}^T I U_{k+1}$, and the identity matrix $I$ is diagonally dominant, we see that $X_{k} \in DD(U_{k+1})$ and hence is feasible for iteration $k+1$. This entails that the optimal value at iteration $k+1$ is at least as good as the optimal value at the previous iteration; i.e., $DSOS_{k+1}\leq DSOS_k$. Since the sequence $\{DSOS_k\}$ is lower bounded by $SOS^*$ and monotonic, it must converge to a limit $DSOS^*\geq SOS^*$. We have been unable to formally rule out the possibility that $DSOS^*>SOS^*$. In all of our numerical experiments, convergence to $SOS^*$ happens (i.e., $DSOS^*=SOS^*$), though the speed of convergence seems to be problem dependent (contrast e.g. the results of Section~\ref{sec:StableSet} with Section~\ref{sec:Partition}). What is easy to show, however, is that if $X_k$ is positive definite\footnote{This would be the case whenever our inner approximation is not touching the boundary of the psd cone in the direction of the objective. As far as numerical computation is concerned, this is of course always the case.}, then the improvement from step $k$ to $k+1$ is actually \emph{strict}.

\begin{theorem}\label{thm:strict.imp}
Let $X_k$ (resp. $X_{k+1}$) be an optimal solution of iterate $k$ (resp. $k+1$) of (\ref{eq:LPChol}) and assume that $X_k$ is pd and $SOS^*<DSOS_{k}$. Then, $$DSOS_{k+1}<DSOS_{k}.$$
\end{theorem}

\begin{proof}
We show that for some $\lambda\in(0,1)$, the matrix $\hat{X}\mathrel{\mathop{:}}= (1-\lambda)X_k+\lambda X^*$ is feasible to the LP in iteration number $k+1$. We would then have that $$DSOS_{k+1}\leq C \cdot \hat{X}=(1-\lambda)C \cdot X_k+\lambda C \cdot X^*<{\gh DSOS_k,}$$
as we have assumed that $C \cdot X^*=SOS^*<DSOS_k=C \cdot X_k.$ To show feasibility of $\hat{X}$ to LP number $k+1$, note first that as both $X_k$ and $X^*$ satisfy the affine constraints $A_i \cdot X=b_i$, then $\hat{X}$ must also. Since $X_k=U_{k+1}^TU_{k+1}$ and $X_k$ is pd,  $U_{k+1}$ must have positive diagonal entries and is invertible. Let $$X_{k+1}^*\mathrel{\mathbb:}=U_{k+1}^{-T}X^*U_{k+1}^{-1}.$$ For $\lambda$ small enough the matrix $(1-\lambda)I+\lambda X_{k+1}^*$ will be dd since we know the identity matrix is strictly diagonally dominant. Hence, the matrix $$\hat{X}=U_{k+1}^T( (1-\lambda)I+\lambda X_{k+1}^*) U_{k+1}$$ is feasible to LP number $k+1$.
\end{proof}




A few remarks are in order. First, instead of the Cholesky decomposition, we could have worked with some other decompositions such as the LDL decomposition $X_k=LDL^T$ or the spectral decomposition $X_k=H^T\Lambda H$ (where $H$ has the eigenvectors of $X_k$ as columns). Aside from the efficiency of the Cholesky decomposition, the reason we made this choice is that the decomposition allows us to write $X_k$ as $U^TIU$ and the identity matrix $I$ is at the analytic center of the set of diagonally dominant matrices {\gh \cite[Section 8.5.3]{BoydBook}}. Second, the reader should see that feasibility of the first LP implies that all future LPs are feasible and lower bounded. While in most applications that we know of the first LP is automatically feasible (see, e.g., the stable set problem in Section~\ref{sec:StableSet}), sometimes the problem needs to be modified to make this the case. An example where this happens appears in Section~\ref{sec:Partition} (see Theorem~\ref{thm:feas.dsos}), where we apply an SOS relaxation to the partition problem.

Alternatively, one can first apply our iterative procedure to a Phase-I problem
\begin{equation} \label{eq:PhaseI}
\begin{aligned}
\alpha_k &\mathrel{\mathop:}= \min \alpha\\
&\text{s.t. } A_i \cdot X=b_i, i=1,\ldots,m\\
&X+\alpha I \in DD(U_k),
\end{aligned}
\end{equation}
with $U_k$ defined as in (\ref{eq:defUk}). Indeed, for $\alpha$ large enough, the initial problem in (\ref{eq:PhaseI}) (i.e., with $U_0=I$) is feasible. Thus all subsequent iterations are feasible and continually decrease $\alpha$. If for some iteration $k$ we get $\alpha_k \leq 0,$ then we can start the original LP sequence (\ref{eq:LPChol}) with the matrix $U_k$ obtained from the last iteration of the Phase-I algorithm.

In an analogous fashion, we can construct  a sequence of SOCPs that provide upper bounds on $SOS^*$. This time, we define a family of cones $${\gh SDD(U)\mathrel{\mathop{:}}=\{M \in S_n ~|~ M=U^TQU, \text{ for some sdd matrix } Q\},}$$ parameterized again by an $n\times n$ matrix $U$. For any $U$, optimizing over the set $SDD(U)$ is an SOCP and we have $SDD(U)\subseteq P_n$. This leads us to the following iterative SOCP sequence:
\begin{equation} \label{eq:SOCPchol}
\begin{aligned} 
SDSOS_k  &\mathrel{\mathop{:}}= \min C \cdot X\\
&\text{s.t. } A_i \cdot X=b_i, i=1,\ldots,m,\\
&X \in SDD(U_k).
\end{aligned}
\end{equation}
Assuming existence of an optimal solution $X_k$ at each iteration, we can once again define the sequence $\{U_k\}$ iteratively as
\begin{align*}
U_0&=I\\
U_{k+1}&=\text{chol}(X_k).
\end{align*}

{\gh The previous statements concerning strict improvement of the LP sequence as described in Theorem \ref{thm:strict.imp}, as well as its convergence} carry through for the SOCP sequence. In our experience, our SOCP bounds converge to the SDP optimal value often faster than our LP bounds do. While it is always true that $SDSOS_0\leq DSOS_0$ (as $DD\subseteq SDD$), the inequality can occasionally reverse in future iterations.  


\begin{figure}[h!]
	\begin{center}
		\mbox{
			\subfigure[LP inner approximations]
			{\label{subfig:DDCones1D}\scalebox{0.46}{\includegraphics{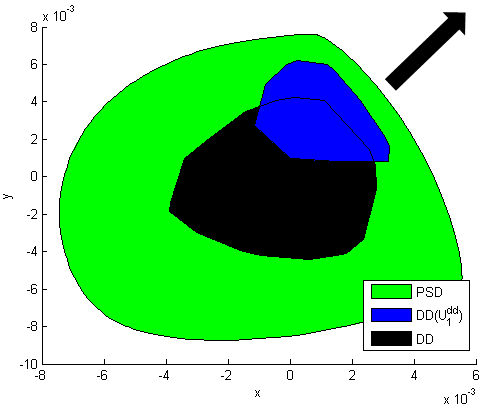}}}}
		\mbox{
			\subfigure[SOCP inner approximations]
			{\label{subfig:SDDCones1D}\scalebox{0.45}{\includegraphics{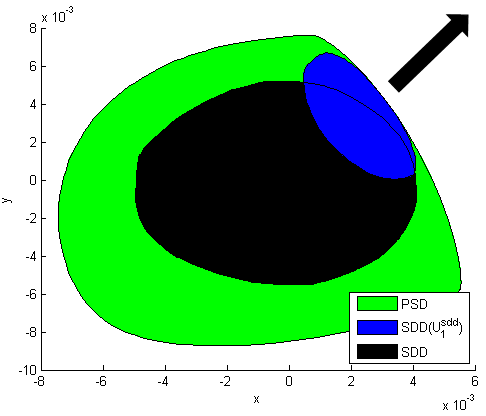}}}
		}
		
		\caption{Improvement after one Cholesky decomposition when maximizing the objective function $x+y$}
		\label{fig:DDSDDCones1D}
	\end{center}
\end{figure}

An illustration of both procedures is given in Figure \ref{fig:DDSDDCones1D}. We generated two random symmetric matrices $A$ and $B$ of size $10 \times 10$. The outermost set is the feasible set of an SDP with the constraint $I+xA+yB \succeq 0$. The goal is to maximize the function $x+y$ over this set. The set labeled $DD$ in Figure \ref{subfig:DDCones1D} (resp. $SDD$ in Figure \ref{subfig:SDDCones1D}) consists of the points $(x,y)$ for which $I+xA+yB$ is dd (resp. sdd). Let ($x_{dd}^*,y_{dd}^*$) (resp. ($x_{sdd}^*,y_{sdd}^*$)) be optimal solutions to the problem of maximizing $x+y$ over these sets. The set labeled $DD(U_1^{dd})$ in Figure \ref{subfig:DDCones1D} (resp. $SDD(U_1^{sdd})$ in Figure \ref{subfig:SDDCones1D}) consists of the points $(x,y)$ for which $I+xA+yB \in DD(U_1^dd)$ (resp. $\in SDD(U_1^{sdd}$)) where $U_1^{dd}$ (resp. $U_1^{sdd}$) corresponds to the Cholesky decomposition of $I+x_{dd}^*A+y_{dd}^*B$ (resp. $I+x_{sdd}^*A+y_{sdd}^*B$). Notice the interesting phenomenon that while the new sets happen to shrink in volume, they expand in the direction that we care about. Already in one iteration, the SOCP gives the perfect bound here.


%

\begin{figure}[h!]
	\begin{center}
		\mbox{
			\subfigure[LP inner approximations]
			{\label{subfig:DDCones}\scalebox{0.46}{\includegraphics{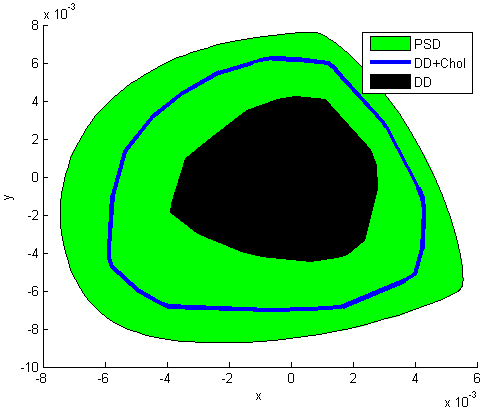}}}}
		\mbox{
			\subfigure[SOCP inner approximations]
			{\label{subfig:SDDCones}\scalebox{0.45}{\includegraphics{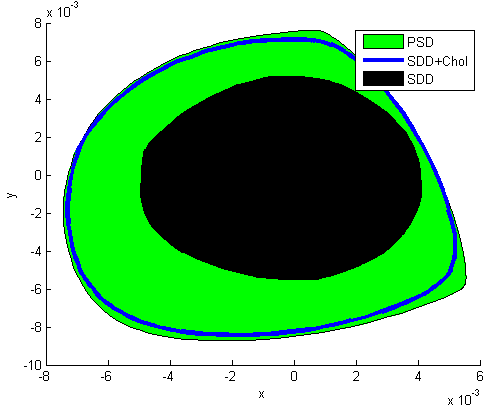}}}
		}
		
		\caption{Improvement in all directions after one Cholesky decomposition}
		\label{fig:DDSDDCones}
	\end{center}
\end{figure}

In Figure \ref{subfig:DDCones}, instead of showing the improvement in just the North-East direction, we show it in all directions. This is done by discretizing a large set of directions $d_i=(d_{i,x},d_{i,y})$ on the unit circle and optimizing along them. More concretely, for each $i$, we maximize $d_{i,x}x+d_{i,x}y$ over the set $I+xA+yB \in DD_n$. We extract an optimal solution every time and construct a matrix $U_{1,d_i}$ from its Cholesky decomposition. We then maximize in the same direction once again but this time over the set $I+xA+yB \in DD(U_{1,d_i})$. The set of all new optimal solutions is what is plotted with the thick blue line in the figure. We proceed in exactly the same way with our SOCPs to produce Figure~\ref{subfig:SDDCones}. Notice that both inner approximations after one iteration improve substantially. The SOCP in particular fills up almost the entire spectrahedron.


\subsection{Inner approximations to the cone of nonnegative polynomials}\label{subsec:polyopt}


A problem domain where inner approximations to semidefinite programs can be useful is in sum of squares programming. This is because the goal of SOS optimization is already to inner approximate the cone of nonnegative polynomials. So by further inner approximating the SOS cone, we will get bounds in the same direction as the SOS bounds.

Let $z(x)$ be the vector of monomials of degree up to $d$. Define a family of cones of degree-$2d$ polynomials $${\gh DSOS(U)\mathrel{\mathop:}=\{p ~|~ p(x)=z^T(x)U^TQUz(x), \text{ for some dd matrix } Q \},}$$ parameterized by an $n\times n$ matrix $U$. We can think of this set as the cone of polynomials that are dsos in the basis $Uz(x)$. If an SOS program has a constraint ``$p$ sos'', we will replace it iteratively by the constraint $p \in DSOS(U_k)$. The sequence of matrices $\{U_k\}$ is again defined recursively with
\begin{align*}
U_0&=I\\
U_{k+1}&=\mbox{chol}(U_k^TQ_kU_k),
\end{align*}
where $Q_k$ is an optimal Gram matrix of iteration $k$. 

Likewise, let $${\gh SDSOS(U)\mathrel{\mathop:}=\{p~|~ p(x)=z(x)^TU^TQUz(x), \text{ for some sdd matrix } Q \}.}$$  This set can also be viewed as the set of polynomials that are sdsos in the basis $Uz(x)$. To construct a sequence of SOCPs that generate improving bounds on the sos optimal value, we replace the constraint $p$ sos by $p \in SDSOS(U_k)$, where $U_k$ is defined as above.

\begin{figure}[h!]
	\begin{center}
		\mbox{
			\subfigure[LP inner approximations]
			{\label{subfig:DSOSCones}\scalebox{0.46}{\includegraphics{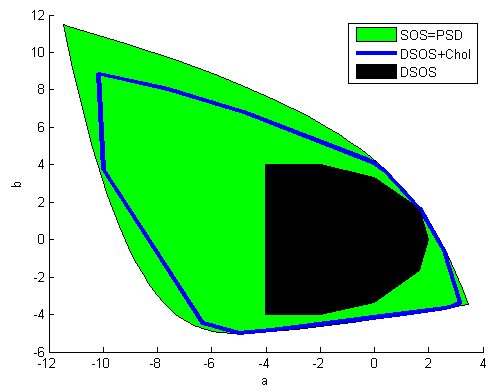}}}}
		\mbox{
			\subfigure[SOCP inner approximations]
			{\label{subfig:SDSOSCones}\scalebox{0.46}{\includegraphics{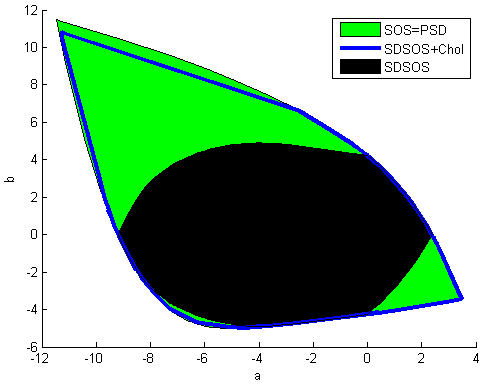}}}
		}
		
		\caption{Improvement in all directions after one Cholesky decomposition}
		\label{fig:DSOSSDSOSCones}
	\end{center}
\end{figure}

In Figure \ref{fig:DSOSSDSOSCones}, we consider a parametric family of polynomials
$$p_{a,b}(x_1,x_2)=2x_1^4+2x_2^4+ax_1^3x_2+(1-a)x_2^2x_2^2+bx_1x_2^3.$$ The outermost set in both figures corresponds to the set of {\gh pairs} $(a,b)$ for which $p_{a,b}$ is sos. As $p_{a,b}$ is a bivariate quartic, this set coincides with the set of $(a,b)$ for which $p_{a,b}$ is nonnegative. The innermost sets in the two subfigures correspond to $(a,b)$ for which $p_{a,b}$ is dsos (resp. sdsos). The {\gh thick blue} lines illustrate the optimal points achieved when maximizing in all directions over the sets obtained from a single Cholesky decomposition. (The details of the procedure are exactly the same as Figure~\ref{fig:DDSDDCones}.) Once again, the inner approximations after one iteration improve substantially over the DSOS and SDSOS approximations.

 
\subsection{Extreme-ray interpretation of the change of basis} \label{subsec:Corners}
In this section, we present an alternative but equivalent way of expressing the LP and SOCP-based sequences. This characterization is based on the extreme-ray description of the cone of diagonally dominant/scaled diagonally dominant matrices. It will be particularly useful when we consider outer approximations of the psd cone in Section \ref{subsec:OuterApprox}.

\begin{lemma}[Barker and Carlson \cite{dd_extreme_rays}]\label{lem:dd.corners}
A symmetric matrix $M$ is diagonally dominant if and only if it can be written as $$M=\sum_{i=1}^{n^2} \alpha_i v_iv_i^T, \alpha_i\geq 0,$$ where $\{v_i\}$ is the set of all {\gh nonzero} vectors in $\mathbb{R}^n$ with at most $2$ nonzero components, each equal to $\pm 1$.
\end{lemma}
The vectors $v_i$ are the extreme rays of the $DD_n$ cone. This characterization of the set of diagonally dominant matrices leads to a convenient description of the dual cone:
\begin{equation}\label{eq:DD*}
DD^*_n=\{X \in {\gh S_n}~|~ v_i^TXv_i \geq 0, i=1,\ldots,n^2\},
\end{equation}
which we will find to be useful in the next subsection. Using Lemma~\ref{lem:dd.corners}, we can rewrite {\gh the sequence of LPs given in (\ref{eq:LPChol}) as}
\begin{equation}\label{eq:LPCorners}
\begin{aligned}
DSOS_k &\mathrel{\mathop{:}}=\min_{X,\alpha_i} C \cdot X\\
&\text{s.t. } A_i \cdot X=b_i,\  i=1,\ldots,m,\\
&X=\sum_{i=1}^{n^2} \alpha_i (U_k^Tv_i)(U_k^Tv_i)^T,\\
&{\gh \alpha_i \geq 0, i=1,\ldots,n^2.}
\end{aligned}
\end{equation}
Let $X_k$ be an optimal solution to the LP in iteration $k$. The sequence of matrices $\{U_k\}$ is defined just as before:
\begin{align*}
U_0&=I\\
U_{k+1}&=\text{chol}(X_k).
\end{align*}

In the first iteration, a linear map is sending (or intuitively ``rotating'') the extreme rays $\{v_iv_i^T\}$ of the dd cone to a new set of extreme rays $\{(U_1^Tv_i)(U_1^Tv_i)^T\}$. This procedure keeps repeating itself without ever changing the number of extreme rays. 



As the sequence of LPs defined in (\ref{eq:LPCorners}) is equivalent to the sequence defined in (\ref{eq:LPChol}), the optimal value of (\ref{eq:LPCorners}) improves in each iteration. This can be seen directly: Indeed, $X_{k}$ is feasible for iteration $k+1$ of (\ref{eq:LPCorners}) by taking $\alpha_i=1$ when {\gh $v_i$ has} exactly one nonzero entry equal to $1$ and $\alpha_i=0$ otherwise. This automatically implies that $DSOS_{k+1}\leq DSOS_k.$ Moreover, the improvement is strict under the assumptions of Theorem~\ref{thm:strict.imp}.

The set of scaled diagonally dominant matrices can be described in a similar fashion. In fact, from (\ref{eq:SDSOS.def}), we know that any scaled diagonally dominant matrix $M$ can be written as $${\gh M=\sum_{i=1}^{ {n \choose 2} }V_i \Lambda_i V_i^T,}$$
where $V_i$ is an $ n \times 2$ matrix whose columns each contain exactly one nonzero element which is equal to $1,$ and $\Lambda_i$ is a $2 \times 2$ symmetric psd matrix.

This characterization of $SDD_n$ gives an immediate description of the dual cone $$SDD^*_n={\gh \left\{X \in S^n ~|~ V_i^TXV_i \succeq 0, i=1,\ldots, {n \choose 2}\right\}},$$ which will become useful later. Our SOCP sequence in explicit form is then
\begin{equation}\label{eq:socp.corner.repres}
\begin{aligned}
SDSOS_k  &= \min_{X,\Lambda_i} C \cdot X\\
&\text{s.t. } A_i \cdot X=b_i, i=1,\ldots,m, \\
&X =\sum_{i=1}^{n \choose 2} (U_k^TV_i) \Lambda_i (U_k^TV_i)^T,\\
&\Lambda_i \succeq 0.
\end{aligned}
\end{equation}

If $X_k$ is an optimal solution at step $k$, the matrix sequence $\{U_k\}$ is defined as before:
\begin{align*}
U_0&=I\\
U_{k+1}&=\text{chol}(X_k).
\end{align*}

The interpretation of (\ref{eq:socp.corner.repres}) is similar to that of (\ref{eq:LPCorners}).

\subsection{Outer approximations of the psd cone}\label{subsec:OuterApprox}

In Section \ref{subsec:InnerApprox}, we considered inner approximations of the psd cone to obtain upper bounds on (\ref{eq:genericSDP}). In many applications, semidefinite programming is used as a ``relaxation'' to provide outer approximations to some nonconvex sets. This approach is commonly used for relaxing quadratic programs; see, e.g., Section~\ref{sec:StableSet}, where we consider the problem of finding the largest stable set of a graph. In such scenarios, it does not make sense for us to inner approximate the psd cone: to have a valid relaxation, we need to outer approximate it. {\gh This can be easily achieved by working with the dual problems, which we will derive explicitly in this section.}


Since $P_n \subseteq DD^*_n$, the first iteration in our LP sequence for outer approximation will be
\begin{align*}
DSOSout_0 &\mathrel{\mathop{:}}= \min_{X} C \cdot X\\
&\text{s.t. } A_i \cdot X =b_i, i=1,\ldots,m,\\
&X \in DD_n^*.
\end{align*}
By the description of the dual cone in (\ref{eq:DD*}), we know this can be equivalently written as
\begin{equation} \label{eq:LP0Outer}
 \begin{aligned}
DSOSout_0&= \min_{X} C \cdot X\\
&\text{s.t. } A_i \cdot X =b_i, \forall i\\
&v_i^TXv_i \geq 0, i=1,\ldots,n^2,
\end{aligned}
\end{equation}
where the $v_i$'s are the extreme rays of the set of diagonally dominant matrices as described in Section \ref{subsec:Corners}; namely, all vectors with at most two nonzero elements which are either $+1$ or $-1$. Recall that when we were after inner approximations (Subsection~\ref{subsec:InnerApprox}), the next LP in our sequence was generated by replacing the vectors $v_i$ by $U^Tv_i$, where the choice of $U$ was dictated by a Cholesky decomposition of an optimal solution of the previous iterate. 
In the outer approximation setting, we seemingly do not have access to a psd matrix that would provide us with a Cholesky decomposition.  However, we can simply get this from the dual of (\ref{eq:LP0Outer})
\begin{align*}
DSOSout^d_0 &\mathrel{\mathop{:}}= \max_{y,\alpha} b^Ty\\
&\text{s.t. } C- \sum_{i=1}^m y_iA_i=\sum_{i=1}^{n^2} \alpha_iv_iv_i^T,\\
&\alpha_i\geq 0, i=1,\ldots,n^2,
\end{align*}
by taking $U_1=\text{chol}(C-\sum_i y_i^*A_i)$. We then replace $v_i$ by $U_1^Tv_i$ in (\ref{eq:LP0Outer}) to get the next iterate and proceed. In general, the sequence of LPs can be written as
\begin{align*}
DSOSout_k&= \min_{X} C \cdot X\\
&\text{s.t. } A_i \cdot X =b_i, i=1,\ldots,m,\\
&v_i^TU_kXU_k^Tv_i \geq 0,
\end{align*}
where $\{U_k\}$ is a sequence of matrices defined recursively as
\begin{align*}
U_0 &=I\\
U_k &=\text{chol}\left(C-\sum_i {\gh y_i^{(k-1)}} A_i \right).
\end{align*}
The vector $y^{k-1}$ here is an optimal solution to the dual problem at step $k-1$:
\begin{align*}
DSOSout^d_{k-1} &\mathrel{\mathop{:}}= \max_{y,\alpha} b^Ty\\
&\text{s.t. } C- \sum_{i=1}^m y_iA_i=\sum_{i=1}^{n^2} \alpha_i(U_{k-1}^Tv_i)(U_{k-1}^Tv_i)^T,\\
&\alpha_i\geq 0, i=1,\ldots,n^2.
\end{align*}
This algorithm again strictly improves the objective value at each iteration. Indeed, from LP strong duality, we have $$DSOSout_k=DSOSout_k^d$$ and Theorem \ref{thm:strict.imp} applied to the dual problem states that $$DSOSout_{k-1}^d<DSOSout_{k}^d.$$

The sequence of SOCPs for outer approximation can be constructed in an analogous manner:
\begin{align*}
SDSOSout_k&= \min_{X} C \cdot X\\
&\text{s.t. } A_i \cdot X =b_i, i=1,\ldots,m,\\
&V_i^TU_kXU_k^TV_i \succeq 0, i=1,\ldots,{n \choose 2},
\end{align*}
where $V_i$'s are $n \times 2$ matrices containing exactly one 1 in each column, and $\{U_k\}$ is a sequence of matrices defined as
\begin{align*}
U_0 &=I\\
U_k &=\text{chol}\left(C-\sum_i \gh{y_i^{(k-1)}} A_i \right)
\end{align*}
Here again, the vector {\gh$y^{(k-1)}$} is an optimal solution to the dual SOCP at step $k-1$:
\begin{align*}
SDSOSout^d_{k-1} &\mathrel{\mathop{:}}= \max_{y,\Lambda_i} b^Ty\\
&\text{s.t. } C- \sum_{i=1}^m y_iA_i=\sum_{i=1}^{n \choose 2} (U_{k-1}^Tv_i)\Lambda_i(U_{k-1}^Tv_i)^T,\\
&\Lambda_i\succeq 0, i=1,\ldots,{n \choose 2},
\end{align*}
where each $\Lambda_i$ is a $2 \times 2$ unknown symmetric matrix.


\begin{remark}
Let us end with some concluding remarks about our algorithm. There are other ways of improving the DSOS and SDSOS bounds. For example, Ahmadi and Majumdar~\cite{iSOS_journal, dsos_cdc14} propose the requirement that $(\sum_{i=1}^n x_i^2)^rp(x)$ be dsos or sdsos as a sufficient condition for nonnegativity of $p$. As $r$ increases, the quality of approximation improves, although the problem size also increases very quickly. Such hierarchies are actually commonly used in the sum of squares optimization literature. But unlike our approach, they do not take into account a particular objective function and may improve the inner approximation to the PSD cone in directions that we do not care about.
%
%
Nevertheless, these hierarchies have interesting theoretical implications. Under some assumptions, one can prove that as $r\rightarrow\infty$, the underlying convex programs succeed in optimizing over the entire set of nonnegative polynomials; see, e.g.,~\cite{Reznick_Unif_denominator, jesus_polya, PhD:Parrilo, iSOS_journal}.

Another approach to improve on the DSOS and SDSOS bounds appears in the recent work in~\cite{isos_cg} {\gh with Dash. We show there how ideas from column generation in large-scale integer and linear programming can be used to iteratively improve inner approximations to semidefinite cones. The LPs and SOCPs proposed in that work take the objective function into account and increase the problem size after each iteration by a moderate amount.} By contrast, the LPs and SOCPs coming from our Cholesky decompositions in this paper have exactly the same size in each iteration. We should remark however that the LPs from iteration two and onwards are typically more dense than the initial LP (for DSOS) and slower to solve. A worthwhile future research direction {\gh would be to systematically compare the performance of the two approaches and to explore customized solvers for the LPs and the SOCPs that arise in our algorithms.}

\end{remark}

\section{The maximum stable set problem} \label{sec:StableSet}
A classic problem in discrete optimization is that of finding the stability number of a graph. The graphs under our consideration in this section are all undirected and unweighted. A \emph{stable set} (or \emph{independent set}) of a graph $G=(V,E)$ is a set of nodes of $G$ no two of which are adjacent. The stability number of {\gh $G$}, often denoted by $\alpha(G)$, is the size of its maximum stable set(s). The problem of determining $\alpha$ has many applications in {\gh scheduling (see, e.g., \cite{golumbic2005algorithmic}) and coding theory \cite{Lovasz}}. As an example, the maximum number of final exams that can be scheduled on the same day at a university without requiring any student to take two exams is given by the stability number of a graph. This graph has courses IDs as nodes and an edge between two nodes if and only if there is at least one student registered in both courses. Unfortunately, the problem of testing whether $\alpha(G)$ is greater than a given integer $k$ is well known to be NP-complete~\cite{Karp}. Furthermore, the stability number cannot be approximated within a factor $|V|^{1-\epsilon}$ for any $\epsilon >0$ unless P$=$NP \cite{HaastadJohan}.

A straightforward integer programming formulation of $\alpha(G)$ is given by
\begin{align*}
\alpha(G) &=\max \sum_i x_i\\
&\text{s.t. } x_i+x_j \leq 1, \text{ if } {\gh \{i,j\}} \in E\\
&x_i \in \{0,1\}.
\end{align*} 
The standard LP relaxation for this problem is obtained by changing the binary constraint $x_i \in \{0,1\}$ to the linear constraint $x_i \in [0,1]$:
\begin{equation}\label{eq:standard.LP}
\begin{aligned}
LP &\mathrel{\mathop{:}}=\max \sum_i x_i\\
&\text{s.t. } x_i+x_j \leq 1, \text{ if } {\gh \{i,j\}} \in E\\
&x_i \in [0,1].
\end{aligned} 
\end{equation}
Solving this LP results in an upper bound on the stability number. The quality of this upper bound can be improved by adding the so-called \emph{clique inequalities}. The set of $k$-clique inequalities, denoted by $C_k$, is the set of constraints of the type $x_{i_1}+x_{i_2}+\ldots+x_{i_k}\leq 1$, if $(i_1,\ldots,i_k)$ form a clique (i.e., a complete subgraph) of $G$. Observe that these inequalities must be satisfied for binary solutions to the above LP, but possibly not for fractional ones. Let us define a family of LPs indexed by $k$:
\begin{equation}\label{eq:standard.LP.with.clique.inequ}
\begin{aligned}
LP^k &\mathrel{\mathop{:}}=\max \sum_i x_i\\
&x_i \in [0,1]\\
&C_1,\ldots,C_k \text{ are satisfied.}
\end{aligned} 
\end{equation}
Note that $LP=LP^2$ by construction and $\alpha(G)\leq LP^{k+1}\leq LP^k$ for all $k$. We will be comparing the bound obtained by some of these well-known LPs with those achieved via the new LPs that we propose further below.

A famous semidefinite programming based upper bound on the stability number is due to Lov\'{a}sz~\cite{Lovasz}:
\begin{align*}
\vartheta(G) &\mathrel{\mathop{:}}= \max_X J \cdot X\\
&\text{s.t. } I \cdot X=1\\
&X_{ij}=0, ~\forall {\gh \{i,j\}} \in E\\
&X \succeq 0,
\end{align*}
where $J$ here is the all ones matrix and $I$ is the {\gh identity matrix.} The optimal value $\vartheta(G)$ is called the Lov\'{a}sz theta number of the graph. We have the following inequalities $$\alpha(G)\leq \vartheta(G) \leq LP^k, \ \forall k.$$

The fact that $\alpha(G)\leq \vartheta(G)$ is easily seen by noting that if $S$ is a stable set of maximum size and $1_S$ is its indicator vector, then the rank-one matrix $\frac{1}{|S|} 1_S 1_S^T$ is feasible to the SDP and gives the objective value $|S|$. The other inequality states that this SDP-based bound is stronger than the aforementioned LP bound even with all the clique inequalities added (there are exponentially many). A proof can be found e.g. in \cite[Section 6.5.2]{LaurentVall}.

Our goal here is to obtain LP and SOCP based sequences of upper bounds on the Lov\'{a}sz theta number. To do this, we construct a series of outer approximations of the set of psd matrices as described in Section \ref{subsec:OuterApprox}. The first bound in the sequence of LPs is given by:
\begin{align*}
DSOS_0(G) &\mathrel{\mathop{:}}= \max_X J \cdot X\\
&\text{s.t. } I \cdot X=1\\
&X_{ij}=0, ~\forall {\gh \{i,j\}} \in E\\
&X \in DD_n^*.
\end{align*}



In view of (\ref{eq:DD*}), this LP can be equivalently written as

\begin{equation} \label{eq:DSOSLovasz}
\begin{aligned}
DSOS_0(G) &= \max_X J \cdot X\\
&\text{s.t. } I \cdot X=1\\
&X_{ij}=0, ~\forall {\gh \{i,j\}} \in E\\
&v_i^TXv_i \geq 0, i=1,\ldots,n^2,
\end{aligned}
\end{equation}
where $v_i$ is a vector with at most two nonzero entries, each nonzero entry being either $+1$ or $-1$. {\gh This LP is always feasible (e.g., with $X=\frac{1}{n}I$). Furthermore, it is bounded above.} Indeed, the last constraints in (\ref{eq:DSOSLovasz}) imply in particular that for all $i,j$, we must have $$X_{i,j}\leq \frac{1}{2}(X_{ii}+X_{jj}).$$
This, together with the constraint $I\cdot X=1$, implies that the objective $J\cdot X$ must remain bounded. As a result, the first LP in our iterative sequence will give a finite upper bound on $\alpha$. 

To progress to the next iteration, we will proceed as described in Section \ref{subsec:OuterApprox}. The new basis for solving the problem is obtained through the dual\footnote{The reader should not be confused to see both the primal and the dual as maximization problems. We can make the dual a minimization problem by changing the sign of $y$.} of (\ref{eq:DSOSLovasz}):
\begin{equation}\label{eq:DSOSdualLovasz}
\begin{aligned}
DSOS_0^d(G) &\mathrel{\mathop{:}}=\max y\\
&\text{s.t. } yI+Y-J=\sum_{i=1}^{n^2} \alpha_i v_i v_i^T\\
& Y_{ij}=0 \text{ if } i=j \text{ or } {\gh \{i,j\}} \notin E\\
&\alpha_i \geq 0, i=1,\ldots,n^2.
\end{aligned}
\end{equation}
The second constraint in this problem is equivalent to requiring that $yI+Y-J$ be dd. We  can define $$U_1=\text{chol}(y_0^*I+Y_0^*-J)$$ where $(y_1^*,Y_1^*)$ are optimal solutions to (\ref{eq:DSOSdualLovasz}). We then solve 
\begin{align*}
DSOS_1(G) &\mathrel{\mathop{:}}= \max_X J \cdot X\\
&\text{s.t. } I \cdot X=1\\
&X_{ij}=0, ~\forall {\gh \{i,j\}} \in E\\
&v_i^TU_1XU_1^Tv_i \geq 0, i=1,\ldots,n^2,
\end{align*}
to obtain our next iterate. The idea remains exactly the same for a general iterate $k$: We construct the dual 
\begin{align*}
DSOS_{k}^d(G) &\mathrel{\mathop{:}}=\max y\\
&\text{s.t. } yI+Y-J=\sum_{i=1}^{n^2} \alpha_i U_{k}^T v_i (U_{k}^T v_i)^T\\
& Y_{ij}=0 \text{ if } i=j \text{ or } {\gh \{i,j\}} \notin E\\
&\alpha_i \geq 0, \forall i,
\end{align*}
and define 
\begin{align*}
U_{k+1}\mathrel{\mathop{:}}=\text{chol}(y_k^*+Y_k^*-J), 
\end{align*}
where $(y_k^*, Y_k^*)$ is an optimal solution to the dual. The updated primal is then
\begin{equation}\label{eq:DSOSLovaszk}
\begin{aligned}
DSOS_{k+1}(G) &\mathrel{\mathop{:}}= \max_X J \cdot X\\
&\text{s.t. } I \cdot X=1\\
&X_{ij}=0, ~\forall {\gh \{i,j\}} \in E\\
&v_i^TU_{k+1}XU_{k+1}^Tv_i \geq 0, i=1,\ldots,n^2.
\end{aligned}
\end{equation} 
As stated in Section \ref{subsec:OuterApprox}, the optimal values of (\ref{eq:DSOSLovaszk}) are guaranteed to strictly improve as a function of $k$. Note that to get the bounds, we can just work with the dual problems throughout. 

An analoguous technique can be used to obtain a sequence of SOCPs. For the initial iterate, instead of requiring that $X \in DD^*$ in (\ref{eq:DSOSLovasz}), we require that $X \in SDD^*$. This problem must also be bounded and feasible as $$P_n\subseteq SDD^*\subseteq DD^*.$$ Then, for a given iterate $k$, the algorithm consists of solving
\begin{align*}
SDSOS_k(G) &\mathrel{\mathop{:}}=\max_X J \cdot X\\
&\text{s.t. } I \cdot X=1\\
&X_{ij}=0, \forall {\gh \{i,j\}} \in E\\
&V_i^TU_kXU_k^TV_i \succeq 0, i=1,\ldots,{n \choose 2},
\end{align*}
where as explained in Section~\ref{subsec:Corners} each $V_i$ is an $ n \times 2$ matrix whose columns contain exactly one nonzero element which is equal to $1$. The matrix $U_k$ here is fixed and obtained by first constructing the dual SOCP
\begin{align*}
SDSOS_{k}^d(G) &\mathrel{\mathop{:}}=\max y\\
&\text{s.t. } yI+Y-J=\sum_{i=1}^{n\choose 2} U_{k}^T V_i \Lambda_i (U_{k}^T V_i)^T\\
& Y_{ij}=0 \text{ if } i=j \text{ or } {\gh \{i,j\}} \notin E\\
&\Lambda_i \succeq 0, \forall i,
\end{align*}
(each $\Lambda_i$ is a symmetric $2\times 2$ matrix decision variable) and then taking $$U_{k}=\text{chol}(y_k^*I+Y_k^*-J).$$

Once again, one can just work with the dual problems to obtain the bounds.





As our first example, we apply both techniques to the problem of finding the stability number of the complement of the Petersen graph (see Figure \ref{subfig:PetersenGraph}). The exact stability number here is 2 and an example of a maximum stable set is illustrated by the two white nodes in Figure \ref{subfig:PetersenGraph}. The Lov\'asz theta number is 2.5 and has been represented by the continuous line in Figure \ref{subfig:IterDSOSSDSOS}. The dashed lines represent the optimal values of the LP and SOCP-based sequences of approximations for 7 iterations. Notice that already within one iteration, the optimal values are within one unit of the true stability number, which is good enough for knowing the exact bound (the stability number is an integer). From the fifth iteration onwards, they differ from the Lov\'asz theta number by only $10^{-2}$.

\begin{figure}[h!]
	\begin{center}
		\mbox{
			\subfigure[Complement of Petersen graph]
			{\label{subfig:PetersenGraph}\scalebox{0.18}{\includegraphics{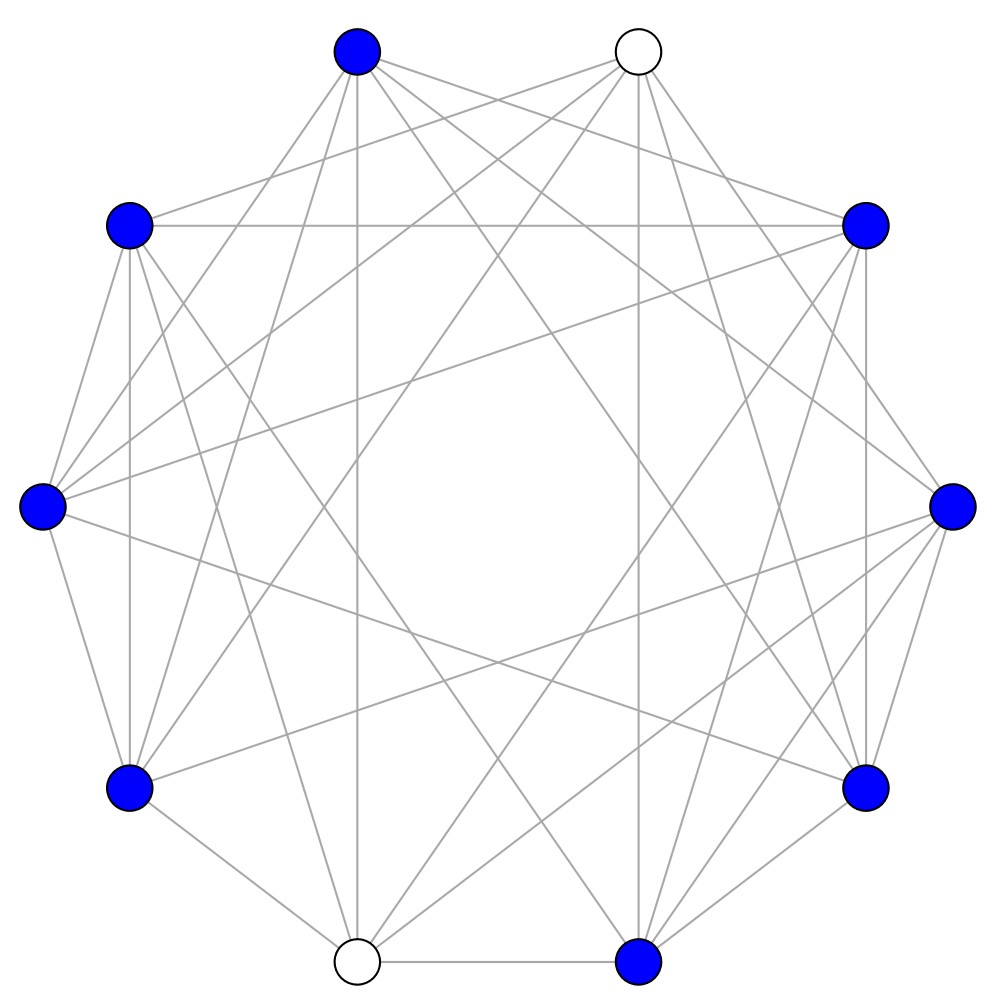}}}}
		\mbox{
			\subfigure[The Lov\'{a}sz theta number and iterative bounds bounds obtained by LP and SOCP]
			{\label{subfig:IterDSOSSDSOS}\scalebox{0.22}{\includegraphics{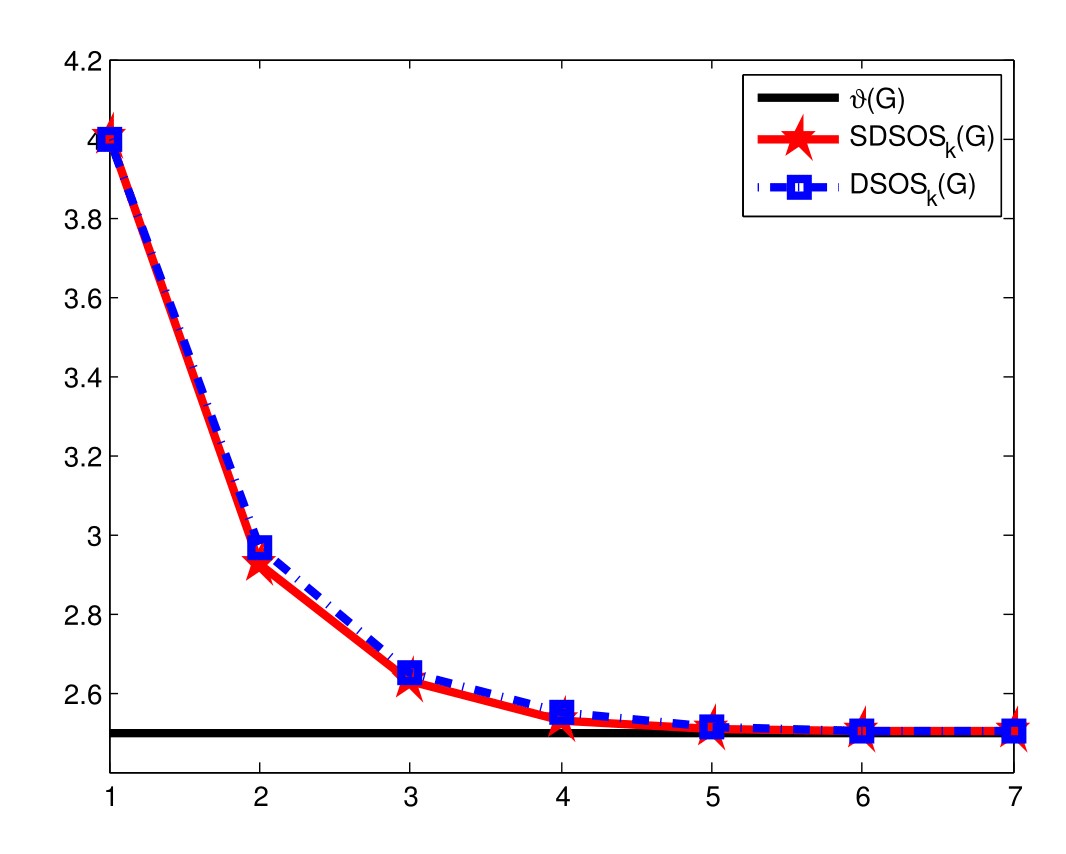}}}
		}
		
		\caption{Upper bounding the stability number of the complement of the Petersen graph}
		\label{fig:dsos.sdsos.Lovasz.Petersen}
	\end{center}
\end{figure}

\vspace{5mm}

Finally, in Table \ref{table:stable_set_success}, we have generated 100 instances of 20-node Erd\"{o}s-R\'{e}nyi graphs with edge probability $0.5$. For each instance, we compute the bounds from the Lov\'{a}sz SDP, the standard LP in (\ref{eq:standard.LP}), the standard LP with all 3-clique inequalities added ($LP^3$ in (\ref{eq:standard.LP.with.clique.inequ})), and our LP/SOCP iterative sequences. We focus here on iterations 3,4 and 5 because there is no need to go further. We compare our bounds with the standard LP and the standard LP with 3-clique inequalities because they are LPs of roughly the same size. If any of these bounds are within one unit of the true stable set number, we count this as a success and increment the counter. As can be seen in Table \ref{table:stable_set_success}, the Lov\'{a}sz theta number is always within a unit of the stable set number, and so are our LP and SOCP sequences ($DSOS_k,SDSOS_k$) after four or at most five iterations. If we {\gh look just} at the bound after 3 iterations, the success rate of SDSOS is noticeably higher than the success rate of DSOS. Also note that the standard LP with or without the three clique inequalities never succeeds in giving a bound within one unit of $\alpha(G)$.\footnote{All numerical experiments in this paper have been parsed using either SPOT~\cite{SPOT_Megretski} or YAMIP~\cite{yalmip} and solved using the LP/SOCP/SDP solver of MOSEK~\cite{mosek}.}


\begin{table}[h!]
\small
\begin{tabular}{|c|c|c|c|c|c|c|c|c|}
  \hline
  $\vartheta(G)$  &$LP$ & $LP^3$ & $DSOS_3$ &$DSOS_4$ & $DSOS_5$ & $SDSOS_3$ & $SDSOS_4$ & $SDSOS_5$\\
  \hline
100\% & 0\% & 0\% & 14\% & 83\% & 100\% & 69\% & 100\% & 100\%\\
  \hline
\end{tabular}
\vspace{2mm}
\caption{Percentage of instances out of 100 where the bound obtained is less than a unit away from the {\gh stability number}}
\label{table:stable_set_success}
\end{table}
\vspace{0.5mm}

\section{Partition}\label{sec:Partition}

The partition problem is arguably the simplest NP-complete problem to state: Given a list of positive integers $a_1,\ldots,a_n$, is it possible to split them into two sets with equal sums? We say that a partition instance is feasible if the answer is yes (e.g., \{5,2,1,6,3,8,5,4,1,1,10\}) and infeasible if the answer is no (e.g., \{47,20,13,15,36,7,46\}). The partition problem is NP-complete but only weakly. In fact, the problem admits a pseudopolynomial time algorithm based on dynamic programming that can deal with rather large problem sizes efficiently. This algorithm has polynomial running time on instances where the bit size of the integers $a_i$ are bounded by a polynomial in $\log n$~\cite{gareyjohnson}. {\gh In this section, we investigate the performance and mostly limitations of algebraic techniques for refuting feasibility of partition instances.}

Feasibility of a partition instance can always be certified by a short proof (the partition itself). However, unless P=co-NP, we do not expect to always have short certificates of infeasibility.
%
Nevertheless, we can try to look for such a certificate through a sum of squares decomposition. Indeed, given an instance $a\mathrel{\mathop{:}}=\{a_1,\ldots,a_n\}$, {\gh it is not hard to see\footnote{ This equivalence is apparent in view of the zeros of the polynomial on the right hand side of (\ref{eq:partition.no.eps}) corresponding to a feasible partition.} that} the following equivalence holds:

\begin{align} \label{eq:partition.no.eps}
{\gh \begin{bmatrix} a \text{ is an infeasible }\\ \text{ partition instance} \end{bmatrix}} \Leftrightarrow p_a(x)\mathrel{\mathop{:}}=\sum_i (x_i^2-1)^2+ (\sum_i a_ix_i)^2>0,~ \forall x\in\mathbb{R}^n.
\end{align}

So if for some $\epsilon >0$ we could prove that $p_a(x)-\epsilon$ is nonnegative, we would have refuted the feasibility of our partition instance.


\begin{definition}\label{def:sos.refut}
An instance of partition $a_1,\ldots,a_n$ is said to be \emph{sos-refutable} if there exists $\epsilon>0$ such that $p_a(x)-\epsilon$ is sos. 
\end{definition}
Obviously, any instance of partition that is sos-refutable is infeasible. This suggests that we can consider solving the following semidefinite program


\begin{equation} \label{eq:sospartition}
\begin{aligned}
SOS &\mathrel{\mathop{:}}=\max \ \epsilon \\
&\text{s.t. } q_a(x) \mathrel{\mathop{:}}=p_a(x)-\epsilon \text{ is sos}
\end{aligned}
\end{equation}
and examining its optimal value. Note that the optimal value of this problem is always greater than or equal to zero as $p_a$ is sos by construction. If the optimal value is positive, we have succeeded in proving infeasibility of the partition instance $a$.


We would like to define the notions of \emph{dsos-refutable} and \emph{sdsos-refutable} instances analogously by replacing the condition $q_a(x)$ sos by the condition $q_a(x)$ dsos or sdsos. Though (\ref{eq:sospartition}) is guaranteed to always be feasible by taking $\epsilon =0$, this is not necessarily the case for dsos/sdsos versions of (\ref{eq:sospartition}). For example, the optimization problem 
\begin{align}\label{eq:dsos.partition.nonhom}
\max_\epsilon\{\epsilon ~|~ p_a(x) -\epsilon  \text{ dsos}\}
\end{align}
 on the instance $\{1,2,2,1,1\}$ is infeasible.\footnote{Under other structures on a polynomial, the same type of problem can arise for sos. For example, consider the Motzkin polynomial~\cite{MotzkinSOS} $M(x_1,x_2)=x_1^2x_2^4+x_2^2x_1^4-3x_1^2x_2^2+1$  which is nonnegative everywhere. The problem $\max_\epsilon\{\epsilon ~|~ M(x) -\epsilon  \text{     sos}\}$ is infeasible.} This is a problem for us as we need the first LP to be feasible to start our iterations. We show, however, that we can get around this issue by modeling the partition problem with homogeneous polynomials.


\begin{definition}
Let $p_a$ be as in (\ref{eq:partition.no.eps}). An instance of partition $a_1,\ldots,a_n$ is said to be \emph{dsos-refutable} (resp. \emph{sdsos-refutable}) if there exists $\epsilon>0$ such that the quartic form
\begin{align}\label{eq:homog}
q_{a,\epsilon}^h(x)\mathrel{\mathop{:}}=p_a \left( \frac{x}{\left(\frac{1}{n}\sum_i x_i^2\right)^{1/2} }\right)\left( \frac{1}{n}\sum_i  x_i^2\right)^2-\epsilon \left( \frac{1}{n} \sum_i x_i^2 \right)^2
\end{align}
is dsos (resp. sdsos).
\end{definition}
Notice that $q_{a,\epsilon}^h$ is indeed a polynomial as it can be equivalently written as
$$ \sum_i x_i^4 +\left(\left (\sum_i a_ix_i\right)^2-2\sum_i x_i^2 \right) \cdot \left( \frac{1}{n}\sum_i  x_i^2\right) +(n-\epsilon) \cdot \left( \frac{1}{n}\sum_i  x_i^2\right)^2.$$ 

What we are doing here is homogenizing a polynomial that does not have odd monomials by multiplying its lower degree monomials with appropriate powers of $\sum_i  x_i^2$. The next theorem tells us how we can relate nonnegativity of this polynomial to feasibility of partition.


\begin{theorem} \label{thm:valid.homog}
A partition instance $a=\{a_1,\ldots,a_n\}$ is infeasible if and only if there exists $\epsilon>0$ for which the quartic form $q_{a,\epsilon}^h(x)$ defined in (\ref{eq:homog}) is nonnegative.
\end{theorem}
\begin{proof}
For ease of reference, let us define 
\begin{align} \label{eq:def.homog.no.eps}
p_a^h(x) \mathrel{\mathop{:}}=p_a \left( \frac{x}{\left(\frac{1}{n}\sum_i x_i^2\right)^{1/2} }\right)\left( \frac{1}{n}\sum_i  x_i^2\right)^2.
\end{align}

Suppose partition is feasible, i.e, the integers $a_1,\ldots,a_n$ can be placed in two sets $\mathcal{S}_1$ and $\mathcal{S}_2$ with equal sums. Let $\bar{x}_i$=1 if $a_i$ is placed in set $\mathcal{S}_1$ and $\bar{x}_i=-1$ if $a_i$ is placed in set $\mathcal{S}_2$. Then $||\bar{x}||_2^2=n$ and $p_a(\bar{x})=0$. This implies that $$p_a^h(\bar{x})=p_a(\bar{x})=0,$$  and hence having $\epsilon>0$ would make $$q_{a,\epsilon}^h(\bar{x})=-\epsilon<0.$$


 If partition is infeasible, then $p_a(x)>0,~ \forall x\in\mathbb{R}^n$. In view of (\ref{eq:def.homog.no.eps}) we see that $p_a^h(x)>0$ on the sphere $\mathbb{S}$ of radius $n$. Since $p_a^h$ is continuous, its minimum $\hat{\epsilon}$ on the compact set $\mathbb{S}$ is achieved and must be positive. So we must have


$$q_{a,\hat{\epsilon}}^h(x)=p_a^h(x)-\hat{\epsilon} \left( \frac{1}{n} \sum_i x_i^2 \right)^2 \geq 0, \forall x \in \mathbb{S}.$$
By homogeneity, this implies that $q_{a,\hat{\epsilon}}^h$ is nonnegative everywhere.
%
\end{proof}

Consider now the LP
\begin{equation} \label{eq:dsos.nh}
\begin{aligned}
&\max_{\epsilon}\  \epsilon\\
&\text{s.t. } q_{a,\epsilon}^h(x) \ \text{dsos.}
\end{aligned}
\end{equation}
\begin{theorem}\label{thm:feas.dsos}
The LP in (\ref{eq:dsos.nh}) is always feasible.
\end{theorem}
\begin{proof}
Let $h(x)\mathrel{\mathop:}=\left(\frac{1}{n}\sum_ix_i^2\right)^2$ and recall that $z(x,2)$ denotes the vector of all monomials of degree exactly $2$. We can write $$h(x)=z^T(x,2)Q_hz(x,2)$$ where $Q_h$ is in the strict interior of the $DD_n$ cone (i.e., its entries $q_{ij}$ satisfy $q_{ii}>\sum_j |q_{ij}|, \forall i$). Furthermore, let $Q$ be a symmetric matrix such that $p_a^h(x)=z(x,2)^TQz(x,2).$ Then
$$q_{a,\epsilon}^h(x)=p_a^h(x)-\epsilon h(x)=z(x,2)^T(Q-\epsilon Q_h)z(x,2).$$
As $Q_h$ is in the strict interior of $DD_n$, $\exists \lambda>0$ such that $$\lambda Q+(1-\lambda) Q_h \text{ is dd.}$$ Taking $\epsilon=-\frac{1-\lambda}{\lambda}$, $Q-\epsilon Q_h$ will be diagonally dominant and $q_{a,\epsilon}^h$ will be dsos. 
\end{proof}
As an immediate consequence, the SOCP
\begin{equation} \label{eq:sdsos.nh}
\begin{aligned}
&\max_{\epsilon} \ \epsilon\\
&\text{s.t. } q_{a,\epsilon}^h(x) \text{ sdsos}
\end{aligned}
\end{equation}
is also always feasible. We can now define our sequence of LPs and SOCPs as we have guaranteed feasibility of the first iteration. This is done following the strategy and notation of Section \ref{subsec:polyopt}:
\begin{equation}\label{eq:LP.sequence}
\begin{aligned}
DSOS_k \text{ (resp. $SDSOS_k$)} &\mathrel{\mathop{:}}= \max_{\epsilon} \ \epsilon\\
&\text{s.t. } q_{a,\epsilon}^h(x) \in DSOS(U_k) \text{ (resp. $SDSOS(U_k)$)},
\end{aligned}
\end{equation}
where $\{U_k\}$ is a sequence of matrices recursively defined with $U_0=I$ and $U_{k+1}$ defined as the Cholesky factor of an optimal dd (resp. sdd) Gram matrix of the optimization problem in iteration $k$.

%

We illustrate the performance of these LP and SOCP-based bounds on the infeasible partition instance $\{1,2,2,1,1\}$. The results are in Figure \ref{fig:instance.12211}. We can use the sum of squares relaxation to refute the feasibility of this instance by either solving (\ref{eq:sospartition}) (the ``non-homogenized version'') or solving (\ref{eq:dsos.nh}) with dsos replaced with sos (the ``homogenized version''). Both approaches succeed in refuting this partition instance, though the homogenized version gives a slightly better (more positive) optimal value. As a consequence, we only plot the homogeneous bound, denoted by $SOS_h$, in Figure~\ref{fig:instance.12211}. Notice that the LP and SOCP-based sequences refute the instance from the $6^{th}$ iteration onwards.

\begin{figure}[h!]
	\begin{center}
		\mbox{
			\subfigure[Bounds $SOS_h$, $DSOS_k$ and $SDSOS_k$]
			{\label{subfig:Partition}\scalebox{0.45}{\includegraphics{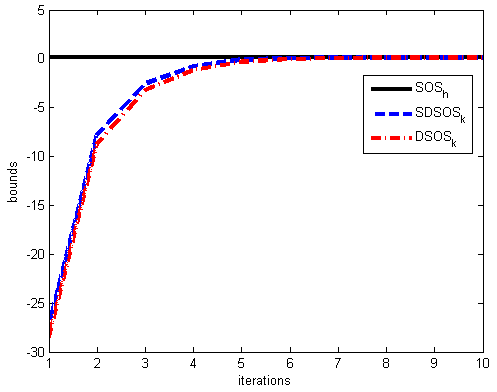}}}}
		\mbox{
			\subfigure[Zoomed-in version of Figure \ref{subfig:Partition}]
			{\label{subfig:PartitionZoom}\scalebox{0.45}{\includegraphics{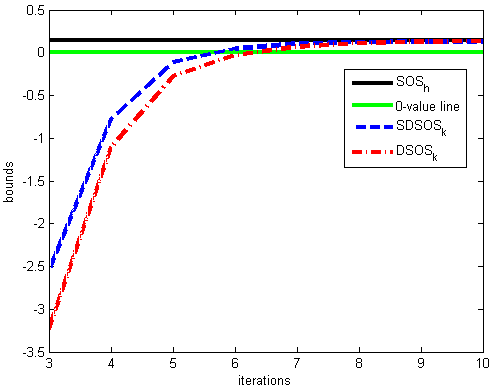}}}
		}
		
		\caption{Bounds obtained on the \{1,2,2,1,1\} instance of the partition problem using SDP, as well as the LP/SOCP-based sequences}
		\label{fig:instance.12211}
	\end{center}
\end{figure}

As our final experiment, we generate 50 infeasible instances of partition with 6 elements randomly generated between 1 and 15. These instances are \emph{trivially infeasible} because we made sure that $a_1+\cdots+a_6$ is an odd number. 
In the first column of Table~\ref{tab:partition.success}, we count the number of successes for sos-refutability (non homogeneous version as defined in Definition \ref{def:sos.refut}), where a failure is defined as the optimal value of (\ref{eq:sospartition}) being 0 up to numerical precision. The second column corresponds to the number of successes for sos-refutability (homogeneous version). The last 4 columns show the success rate of the LP and SOCP-based sequences as defined in (\ref{eq:LP.sequence}), after 20 iterations and 40 iterations.
\begin{table}[H]
\centering
\begin{tabular}{|c|c|c|c|c|c|}
\hline
$SOS$ & $SOS_{h}$ & $DSOS_{20}$ & $DSOS_{40}$ & $SDSOS_{20}$ & $SDSOS_{40}$\\
\hline
56\% & 56\% & 12\% & 16 \% & 14\% & 14\%\\
\hline
\end{tabular}
\caption{Rate of success for refutability of infeasible instances of partition}
 \label{tab:partition.success}
\end{table}
From the experiments, the homogeneous and non-homogeneous versions of (\ref{eq:sospartition}) have the same performance in terms of their ability to refute feasibility. However, we observe that they both fail to refute a large number of completely trivial instances! We prove why this is the case for one representative instance in the next section. The LP and SOCP-based sequences also perform poorly and their convergence is much slower than what we observed for the maximum stable set problem in Section~\ref{sec:StableSet}.

\subsection{Failure of the sum of squares relaxation on trivial partition instances.} \label{subsec:sos.refutability}

For complexity reasons, one would expect there to be infeasible instances of partition that are not sos-refutable. What is surprising however is that the sos relaxation is failing on many instances that are totally trivial to refute as the sum of their input integers is odd. We present a proof of this phenomenon on an instance which is arguably the simplest one.\footnote{If we were to instead consider the instance [1,1,1], sos would succeed in refuting it.}


\begin{proposition} \label{th:infeas.partition.inst}
The infeasible partition instance $\{1,1,1,1,1\}$ is not sos-refutable.
\end{proposition}
\begin{proof}
Let $p_a$ be the polynomial defined in (\ref{eq:partition.no.eps}). To simplify notation, we let $p(x)$ represent $p_a(x)$ for $a=\{1,1,1,1,1\}$. We will show that $p$ is on the boundary of the SOS cone even though we know it is strictly inside the PSD cone. This is done by presenting a dual functional $\mu$ that vanishes on $p,$ takes a nonnegative value on all quartic sos polynomials, and a negative value on $p(x)-\epsilon$ for any $\epsilon>0.$ (See Figure~\ref{fig:proof.partition} for an intuitive illustration of this.)

\begin{figure}[H]
\centering
\includegraphics[scale=0.5]{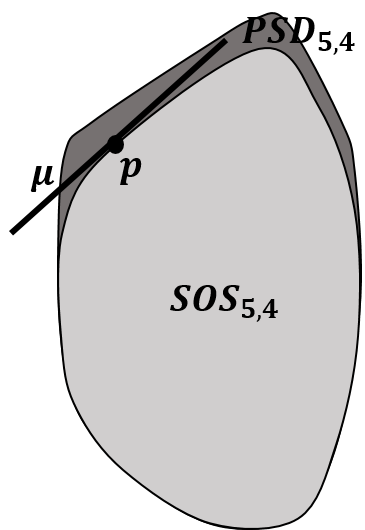}
\caption{The geometric idea behind the proof of Proposition \ref{th:infeas.partition.inst}}
\label{fig:proof.partition}
\end{figure}


The polynomial $p$ when expanded out reads
\begin{align} \label{eq:p.expanded}
p(x)=n-\sum_i x_i^2+2\sum_{i<j} x_i x_j +\sum_i x_i^4.
\end{align}
Consider the vector of coefficients of $p$ with the ordering as written in (\ref{eq:p.expanded}):
\setcounter{MaxMatrixCols}{21}
\scalefont{0.6}
\begin{align} \label{eq:pvec}
\overrightarrow{p}=\begin{pmatrix} 5&-1&-1&-1&-1&-1&2&2&2&2&2&2&2&2&2&2&-1&-1&-1&-1&-1
\end{pmatrix}.
\end{align}
\normalsize
This is a reduced representation of the vector of coefficients of $p$, in that there are many zeros associated with all other monomials of degree less than or equal to 4, which we are not writing out.


Our goal is to find a vector $\mu$ that satisfies
\begin{align}
\langle \mu, \overrightarrow{p} \rangle &=0 \nonumber \\
\langle \mu,\overrightarrow{q} \rangle &\geq 0, \text{ for all $q$ sos of degree 4}. \label{eq:inner.prod.sos.zero}
\end{align}
If such a $\mu$ exists and its first element is nonzero (which by rescaling can then be taken to be 1), then $\langle \mu, \overrightarrow{p-\epsilon} \rangle=\langle \mu,\overrightarrow{p} \rangle -\langle \mu, \overrightarrow{\epsilon} \rangle=-\epsilon<0$. This provides us with the required functional that separates $p(x)-\epsilon$ from the set of sos polynomials.

Selecting the same reduced basis as the one used in (\ref{eq:pvec}), we take
\begin{align*}
\overrightarrow{\mu_{reduced}}=\begin{pmatrix} 1 & \textbf{1}_5^T & -\frac{1}{4} \cdot \textbf{1}_1^T & \textbf{1}_5^T 
\end{pmatrix}
\end{align*}
where $\textbf{1}_n$ is the all ones vector of size $n$. The subscript ``reduced'' denotes the fact that in $\overrightarrow{\mu_{reduced}}$, only the elements of $\mu$ needed to verify $\langle \mu, \overrightarrow{p}\rangle=0$ are presented. Unlike $\overrightarrow{p}$, the entries of $\mu$ corresponding to the other monomials are not all zero. This can be seen from the entries of the matrix $M$ that appears further down.

We now show how (\ref{eq:inner.prod.sos.zero}) holds. Consider any sos polynomial $q$ of degree less than or equal to 4. We know that it can be written as $$q(x)=z^TQz=\mbox{Tr}~ Q \cdot zz^T,$$ for some $Q \succeq 0$, and a vector of monomials 
$$z^T=[1,x_1,x_2,\ldots,x_n,x_1^2,\ldots,x_n^2,x_1x_2, \ldots, x_{n-1}x_n].$$
It is not difficult to see that $$ \langle \mu, \overrightarrow{q} \rangle = \mbox{Tr}~ Q \cdot (zz^T)|_\mu$$
where by $(zz)^T|_\mu$, we mean a matrix where each monomial in $zz^T$ is replaced with the corresponding element of the vector $\mu$. This yields the matrix
\scalefont{1}
\[ M=\begin{pmatrix}
   1&    0&            0 &       0        &     0      &       0         &    1          &   1    &         1   &           1    &         1    &       b &          b &          b &          b &          b &          b &          b &          b &          b &          b  \\    
       0         &    1  &         b &          b &          b &          b &            0     &        0        &     0      &       0     &        0   &          0 &             0       &      0   &          0    &         0        &     0       &      0     &        0       &      0     &        0      \\
       0    &       b &            1 &          b &          b &          b &            0 &            0    &         0      &       0         &    0   &          0            & 0 &            0        &     0          &   0       &      0     &        0   &           0   &           0        &     0      \\
       0   &        b &          b &            1      &     b &          b &            0       &      0          &   0      &       0       &      0       &      0&             0            & 0            & 0           &  0          &   0           &  0          &   0          &   0         &    0      \\
       0  &         b &          b &          b &            1   &        b &            0   &          0   &          0   &          0     &        0  &           0 &            0         &    0          &   0      &       0      &       0        &     0         &    0            & 0         &    0      \\
       0    &       b &          b &          b &          b &            1  &           0     &        0   &          0       &      0  &           0       &      0          &   0        &     0           &  0       &      0          &   0  &           0     &        0        &     0           &  0      \\
       1   &          0   &          0     &        0        &     0      &       0    &         1         &    1       &      1  &           1       &      1  &         b &          b &          b &          b &          b &          b &          b &          b &          b &          b \\     
       1      &       0          &   0        &     0     &        0        &     0     &        1     &        1         &    1       &      1       &      1        &   b &          b &          b &          b &          b &          b &          b &          b &          b &          b    \\  
       1      &       0        &     0    &         0     &        0    &         0       &      1         &    1      &       1           &  1            & 1           &b &          b &          b &          b &          b &          b &          b &          b &          b &          b  \\
       1       &      0       &      0     &        0      &       0      &       0      &       1        &     1        &     1  &           1        &     1    &       b &          b &          b &          b &          b &          b &          b &          b &          b &          b    \\
       1    &         0       &      0      &       0     &        0      &       0    &         1       &      1     &        1     &        1        &     1       &    b &          b &          b &          b &          b &          b &          b &          b &          b &          b    \\
     b &            0    &         0   &          0     &        0         &    0      &     b &          b &          b &          b &          b &            1        &   b &          b &          b &          b &          b &          b &           a &           a &           a\\     
     b &            0    &         0  &           0      &       0      &       0      &     b &          b &          b &          b &          b &          b &            1     &      b &          b &          b &           a &           a &          b &          b &           a   \\  
     b &            0  &           0        &     0        &     0          &   0       &    b &          b &          b &          b &          b &          b &          b &            1    &       b &           a &          b &           a &          b &           a &          b      \\
     b &            0          &   0            & 0      &       0     &        0        &   b &          b &          b &          b &          b &          b &          b &          b &            1          &  a &           a &          b &           a &          b &          b     \\
     b &            0   &          0      &       0       &      0       &      0       &    b &          b &          b &          b &          b &          b &          b &           a &           a &            1        &   b &          b &          b &          b &           a      \\
     b &            0     &        0       &      0      &       0            & 0 &          b &          b &          b &          b &          b &          b &           a &          b &           a &          b &            1    &       b &          b &           a &          b     \\
     b &            0       &      0    &         0   &          0   &          0        &   b &          b &          b &          b &          b &          b &           a &           a &          b &          b &          b &            1     &       a &          b &          b   \\
     b &            0       &      0        &     0      &       0          &   0   &        b &          b &          b &          b &          b &           a &          b &          b &           a &          b &          b &           a &            1       &    b &          b    \\
     b &            0         &    0      &       0      &       0      &       0    &       b &          b &          b &          b &          b &           a &          b &           a &          b &          b &           a &          b &          b &            1        &   b      \\
     b &            0     &        0      &       0     &        0      &       0  &         b &          b &          b &          b &          b &           a &           a &          b &          b &           a &          b &          b &          b &          b &            1      \\
\end{pmatrix},
\]
\normalsize
where $a=\frac{3}{8}$ and $b=-\frac{1}{4}$. We can check that $M \succeq 0$. This, together with the fact that $Q\succeq 0,$ implies that (\ref{eq:inner.prod.sos.zero}) holds.\footnote{It can be shown in a similar fashion that $\{1,1,1,1,1\}$ is not sos-refutable in the homogeneous formulation of (\ref{eq:homog}) either.}
\end{proof}

{\gh
\subsection{Open problems}

We showed in the previous subsection that the infeasible partition instance $\{1,1,1,1,1\}$ was not sos-refutable. Many more randomly-generated partition instances that we knew to be infeasible (their sum being odd) also failed to be sos-refutable. This observation motivates the following open problem:


\paragraph{Open Problem 1} Characterize the set of partition instances $\{a_1,\ldots,a_n\}$ that have an odd sum but are not sos-refutable (see Definition~\ref{def:sos.refut}).

Our second open problem has to do with the power of higher order sos relaxations for refuting feasibility of partition instances.
\paragraph{Open Problem 2} For a positive integer $r$, let us call a partition instance $\{a_1, \ldots, a_n\}$ \emph{$r$-sos-refutable} if $\exists \epsilon>0$ such that $(p(x)-\epsilon)(\sum_i x_i^2+1)^r$ is sos. Note that this is also a certificate of infeasibility of the instance. Even though the $\{1,1,1,1,1\}$ instance is not sos-refutable, it is $1$-sos-refutable. Furthermore, we have numerically observed that the instance $\{1,1,1,1,1,1,1\}$ (vector of all ones of length 7) is not sos-refutable or $1$-sos-refutable, but it is $2$-sos-refutable. If we consider the instance consisting of $n$ ones with $n$ odd, and define $\tilde{r}$ to be the minimum $r$ such that $\{1,1,\ldots,1\}$ becomes $r$-sos-refutable, is it true that $\tilde{r}$ must grow with $n$?
}

\section{Acknowledgments} We are grateful to Anirudha Majumdar and Sanjeeb Dash for insightful discussions. 

\bibliographystyle{amsplain} 
\bibliography{pablo_amirali}

\end{document}